\documentclass[11pt]{article}

\usepackage[utf8]{inputenc}
\usepackage[T1]{fontenc}
\usepackage[english]{babel}

\usepackage[a4paper, top=2.5cm, bottom=2.5cm,
left=3cm, right=3cm, includeheadfoot]{geometry}

\usepackage{amsmath, amssymb, amsfonts, amsthm}

\usepackage{graphicx}
\usepackage{xcolor}

\usepackage{comment}
\usepackage{float}

\usepackage{hyperref}
\hypersetup{
	colorlinks=true,
	linkcolor=blue!60!black,
	citecolor=blue!60!black,
	urlcolor=blue!60!black
}

\newtheorem{thm}{Theorem}[section]
\newtheorem{defi}[thm]{Definition}
\newtheorem{rem}[thm]{Remark}
\newtheorem{prop}[thm]{Proposition}
\newtheorem{lemma}[thm]{Lemma}
\newtheorem{cor}[thm]{Corollary}

\newcommand{\floor}[1]{\left\lfloor #1 \right\rfloor}
\newcommand{\ceil}[1]{\left\lceil #1 \right\rceil}

\title{Christoffel Words as Extremal Structures in Collatz Dynamics}

\author{
	Carlos Fern\'andez\thanks{
		Departamento de Matem\'aticas,
		University of Oviedo,
		E-33007 Oviedo, Spain.\\
		\texttt{carlos@uniovi.es}}
	\and
	Santiago Ib\'a\~nez\thanks{
		Departamento de Matem\'aticas,
		University of Oviedo,
		E-33007 Oviedo, Spain.\\
		\texttt{mesa@uniovi.es}}
}

\date{July 2026}

\begin{document}
	
	\maketitle
	
	\begin{abstract}
		We study the combinatorial structure of parity sequences associated with the accelerated Collatz map with the goal of identifying extremal configurations and relating them to the existence of periodic orbits. To each finite sequence of an orbit, we associate a binary word whose ones encode the odd iterates, and we introduce a functional $C(d)$ on such words which provides an explicit expression for the iterates and characterizes possible periodic cycles. We define a natural rotation action on binary words, compatible with the cyclic structure of periodic orbits, and consider the functional $C_{\min}(d)$ as a canonical representative of each rotation class. In this setting, we formulate and solve a discrete optimization problem on the set of binary words of fixed length and prescribed density.
		
		We prove that Christoffel words are, up to rotation, the unique maximizers of $C_{\min}(d)$ on $D_{N,r}$, the set of binary words of length $N$ with exactly $r$ ones, thereby establishing a direct connection between the dynamics of the Collatz problem and the classical theory of balanced words. As a consequence, we obtain restrictions on the possible existence of nontrivial cycles and derive explicit bounds for the minimum element of an orbit in terms of its length and the proportion of odd iterates. These results show that the combinatorial structure of parity sequences imposes strong constraints on Collatz dynamics and suggest that extremal configurations are governed by classical objects from the combinatorics on words, exhibiting a pronounced structural rigidity.
	\end{abstract}
	
	\section{Introduction}
	
		The Collatz problem, despite its elementary formulation, conceals a rich combinatorial structure. In this paper we take a parity‑sequence viewpoint and relate the binary words that encode Collatz orbits to classical objects such as Christoffel words, which also appear in optimal scheduling, digital geometry, and Sturmian sequences.
	
	Let
	\begin{equation}
		\phi(x) = 
		\begin{cases}
			x/2, & x \equiv 0 \pmod 2,\\
			(3x+1)/2, & x \equiv 1 \pmod 2,
		\end{cases}
		\label{eq:funccollatz}
	\end{equation}
	be defined on the set of positive integers. The Collatz conjecture asserts that for every $x \in \mathbb{N}$ there exists $N \ge 0$ such that $\phi^{N}(x) = 1$. A classical approach to this problem is to
	study the possible existence of periodic orbits other than the trivial 2-cycle $1 \leftrightarrow 2$. Terras \cite{Terras1976} showed that $\phi^N(x)$ can be expressed explicitly in terms of the parity sequence $(d_1,\dots,d_N)$, where $d_i \in \{0,1\}$ records the parity of the $(i-1)$-th iterate: $d_i = 0$ if $\phi^{i-1}(x)$ is even, and $d_i = 1$ if it is odd. This encoding links the study of the dynamics to the combinatorial analysis of binary words.
	
	The dynamics of the Collatz map has been the subject of extensive investigation since the pioneering work of L.~Collatz in the 1930s. Early contributions by Terras and Everett studied stopping times and structural properties of the iteration \cite{Terras1976,Everett1977}, while Lagarias provided a systematic account of the problem and its variants \cite{Lagarias1985,Lagarias2010}. More recent advances include density bounds, probabilistic models, and almost-everywhere results; see, for instance, \cite{ApplegateLagarias2003,KrasikovLagarias2003,Tao2019,KontorovichLagarias2010}.
	
	From a combinatorial viewpoint, it is natural to encode orbits by binary words associated with the parity of the iterates. This perspective has been developed explicitly in recent work on parity vectors: Rajab \cite{Rajab2022} introduces characteristic numbers associated with parity vectors that share a similar flavour with the functional $C(d)$ studied here, while Rozier \cite{Rozier2025} analyzes order relations on binary words that are closely related to the local transformations we use in Section~5. This places the problem in the framework of combinatorics on words and symbolic dynamics, where classical objects such as Sturmian and Christoffel words naturally appear \cite{Lothaire2002,AlloucheShallit2003}.
	
	On the other hand, the analysis of nontrivial cycles has led to the introduction of functionals associated with parity sequences, which allow one to formulate the existence of periodic orbits in terms of precise arithmetic conditions. Such ideas are already implicit in the work of Terras \cite{Terras1976} and have been further developed from different points of view; see, for instance, \cite{Lagarias1985,Rajab2022}.
	
	Problems of this type, involving discrete optimization over binary words of fixed length and density, have been extensively studied in combinatorics on words, where Christoffel words and, more generally, Sturmian words appear as extremal configurations characterized by optimal balance and regularity properties \cite{AlloucheShallit2003,BerstelDeLuca2002,Lothaire2002}. Up to the recent contribution \cite{Knight2026}, a paper we discovered only after our work was completed, and to the best of our knowledge, this connection has not been systematically developed in the specific context of the Collatz problem. However, the explicit form of the functionals associated with parity sequences—in particular, their exponential dependence on the distribution of symbols—justifies the optimization problems that give rise to these structures.
		
	In this context, the present work lies at the intersection of the arithmetic dynamics of the Collatz map and the combinatorics on words. We introduce a functional $C(d)$ associated with parity sequences, define a rotation-invariant version $C_{\min}(d)$, and prove that Christoffel words are, up to rotation, the unique maximizers of $C_{\min}(d)$ on $D_{N,r}$, the set of binary words of length $N$ with exactly $r$ ones. As a consequence, we derive explicit restrictions on the existence of nontrivial periodic orbits, including the bound
	\begin{equation*}
		x \le \frac{1}{2^{N/r}-3}
	\end{equation*}
	(valid for $N/r > \log_2 3$, which is a necessary condition for the existence of a periodic orbit) on the minimum element, $x$, of a cycle of length $N$ with $r$ odd iterates, and show that no nontrivial cycle can satisfy $N \ge 2r$.
	
	The paper is organized as follows. In Section~2 we introduce parity sequences and the functional $C(d)$. Section~3 studies the rotation action and defines $C_{\min}(d)$. Christoffel words are recalled in Section~4. Local transformations and the induced partial order are analyzed in Section~5, and the combinatorial structure of minimizers is developed in Section~6. The main maximization theorem is proved in Section~7, and its consequences for Collatz dynamics are derived in Section~8.
	
	\section{Parity sequences}
	
		Let $\mathcal{D}_N = \{0,1\}^N$ denote the set of binary words of length $N$. Given $0 \le r \le N$, we define the subset
		\begin{equation*}
			\mathcal{D}_{N,r} = \left\{ d = (d_1,\dots,d_N) \in \mathcal{D}_N
			: \sum_{j=1}^N d_j = r \right\},
		\end{equation*}
		that is, the set of all binary words of length $N$ with exactly $r$ ones.
		
		Given $x \in \mathbb{N}$, consider the orbit $x, \phi(x), \phi^2(x), \dots, \phi^{N-1}(x)$, with $\phi$ as given in \eqref{eq:funccollatz}. We associate to it the binary word $d = (d_1,\dots,d_N)$ defined by
		\begin{equation*}
			d_i = \delta_{1(2)}\!\left(\phi^{i-1}(x)\right) \in \{0,1\},
		\end{equation*}
		where $\delta_{1(2)}(n)=0$, if $n$ is even, and $1$, if it is odd. In this way, each initial sequence of length $N$ of a Collatz orbit determines a binary word in $\mathcal{D}_{N,r_0(d)}$.
		
		For each $i = 0, \dots, N$, we define
		\begin{equation*}
			r_i(d) = \sum_{j=i+1}^{N} d_j,
		\end{equation*}
		with the convention $r_N(d) = 0$. Thus $r_0(d) = r$ is the total number of ones in $d$, and $r_i(d)$ counts the ones strictly to the right of position $i$.
	
		Following the classical formula of Terras \cite{Terras1976}, the $N$-th iterate of $\phi$ satisfies
	\begin{equation}\label{eq:Terras}
		\phi^N(x) = \frac{3^{r_0(d)}}{2^N}\,x
		+ \sum_{i=1}^{N} \frac{3^{r_i(d)}}{2^{N-i+1}}\,d_i.
	\end{equation}
	If $x$ belongs to a periodic orbit of period $N$, then
	$\phi^N(x) = x$, and \eqref{eq:Terras} gives
	\begin{equation}\label{eq:cycle}
		x = \frac{C(d)}{2^N - 3^{r_0(d)}},
	\end{equation}
	where
	\begin{equation}\label{eq:C}
		C(d) = \sum_{i=1}^{N} 2^{i-1}\, 3^{r_i(d)}\, d_i.
	\end{equation}
	This representation of the iterates in terms of the parity sequence
	appears already in \cite{Terras1976} and has been used extensively in
	the subsequent literature to translate dynamical properties of the
	Collatz map into arithmetic and combinatorial conditions on binary
	words; see also \cite{Lagarias1985,Rajab2022}.
	
	The quantity $r_i(d)$ admits a natural combinatorial interpretation:
	it counts the number of ones strictly to the right of position $i$,
	that is,
	\begin{equation*}
		r_i(d) = \#\{\, j \in I(d) : j > i \,\},
	\end{equation*}
	where $I(d) = \{\, i \in \{1,\dots,N\} : d_i = 1 \,\}$ is the set
	of positions of the ones in $d$. In particular, $r_i(d)$ can be
	interpreted as the number of odd iterates remaining after step $i$,
	so that the functional $C(d)$ depends on the distribution of the ones
	in $d$ in a global, rather than local, manner.
	
	\section{Rotations}
	
	Suppose that the binary word $d \in \mathcal{D}_{N,r}$ arises from
	a periodic orbit of $\phi$ of period $N$. Since the choice of base
	point along the orbit is not canonical, it is natural to introduce
	an action on $\mathcal{D}_{N,r}$ that models this change of
	reference. We define the \emph{rotation operator}
	$\tau : \mathcal{D}_N \to \mathcal{D}_N$ by
	\begin{equation*}
		\tau(d_1, d_2, \dots, d_N) = (d_2, d_3, \dots, d_N, d_1),
	\end{equation*}
	and denote by $\tau^j$ its $j$-th iterate, with $\tau^0 = \mathrm{id}$.
	The map $\tau$ is a bijection of order dividing $N$, so in particular
	$\tau^N = \mathrm{id}$, and it defines an action of the cyclic group
	$\mathbb{Z}/N\mathbb{Z}$ on $\mathcal{D}_{N,r}$.
	
	If $x$ belongs to a periodic orbit of period $N$ with associated
	parity word $d$, then choosing a different base point along the orbit
	yields the word $\tau^j(d)$ for some $j \in \{0, \dots, N-1\}$.
	Consequently, the words associated with a given cycle form a single
	orbit under the action of $\tau$, and any quantity intended to
	describe intrinsic properties of the cycle must be invariant under
	this action.
	
	The functional $C(d)$ defined in \eqref{eq:C}, however, depends on
	the choice of base point. It is therefore natural to pass to the
	rotation-invariant quantity
	\begin{equation}\label{eq:Cmin}
		C_{\min}(d) = \min_{0 \le j \le N-1} C(\tau^j(d)),
	\end{equation}
	which is well defined on rotation classes and provides a canonical
	representative of the equivalence class of $d$ under $\tau$.
	
	\begin{rem}
		From a dynamical viewpoint, $C_{\min}(d)$ corresponds to choosing
		the base point in the orbit for which the numerator $C(d)$ in the
		expression
		\begin{equation*}
			x = \frac{C(d)}{2^N - 3^r}
		\end{equation*}
		is minimized. Note that $2^N \ne 3^r$ for all $N, r \in \mathbb{N}$ with $r \ge 1$,
		since $\log_2 3$ is irrational, so the denominator is always nonzero. This is a standard procedure in the study of the
		Collatz problem, where quantities associated with parity vectors
		depend on the choice of base point and it is natural to consider
		extremal representatives to obtain well-defined invariants; see,
		for example, \cite{Terras1976,Lagarias1985}.
		
		Now, observe that the expression $x = \frac{C(d)}{2^{N} - 3^{r}}$ can be considered for any binary word $d \in \mathcal{D}_{N,r}$, regardless of whether it actually comes from a periodic orbit of the Collatz map. In particular, if the corresponding value of $x$ is not a positive integer, we must conclude that there is no periodic Collatz
		orbit whose parity pattern is encoded by $d$.
		
		On the other hand, it is clear that $C$ induces an order on $\mathcal{D}_{N,r}$. The question we ask is therefore the following: does there exist a pair $(N,r)$ for which
		\begin{equation*}
			\max_{d \in \mathcal{D}_{N,r}} C_{\min}(d) < 2^{N} - 3^{r}?
		\end{equation*}
		In other words, we seek to characterize those pairs $(N,r)$ for which the quotient
		\begin{equation*}
			\frac{C_{\min}(d)}{2^{N} - 3^{r}}
		\end{equation*}
		is always strictly less than one for every $d \in \mathcal{D}_{N,r}$, and hence no periodic orbit with that parity structure can exist.
	\end{rem}
	
	From a combinatorial viewpoint, the problem consists in studying the
	behaviour of $C_{\min}(d)$ over $\mathcal{D}_{N,r}$. This leads to
	a discrete optimization problem: for fixed $N$ and $r$, determine
	which binary words in $\mathcal{D}_{N,r}$ maximize $C_{\min}(d)$.
	The density $r/N$ plays a particularly relevant role, as it can be
	interpreted as the slope associated with the word, in the sense of
	the theory of Sturmian and Christoffel words
	\cite{Lothaire2002,AlloucheShallit2003}.
	
	\section{Christoffel words}
	
	The Christoffel word of parameters $(N,r)$ can be defined by means
	of several equivalent conventions. In this work we adopt the
	following.
	
	\begin{defi}\label{def:Christoffel}
		Let $N \ge 1$ and $0 \le r \le N$. The \emph{Christoffel word} of
		parameters $(N,r)$ is the binary word $d^{\mathrm{chr}}_{N,r} =
		(d_1, \dots, d_N) \in \mathcal{D}_{N,r}$ defined by
		\begin{equation}\label{eq:Christoffel}
			d_i = \ceil{\frac{i\,r}{N}} - \ceil{\frac{(i-1)\,r}{N}},
			\qquad i = 1, \dots, N.
		\end{equation}
	\end{defi}
	
	This definition is equivalent, up to a rotation (and, in some
	cases, an elementary reflection), to the more standard definition
	based on the floor function; see \cite{Lothaire2002}. We adopt
	\eqref{eq:Christoffel} because it is well suited to the analysis
	of the positions of the ones carried out in Sections~6 and~7.
	
	These words provide a balanced distribution of the ones along the word, characterized by optimal regularity properties with respect to the density $r/N$. More precisely, Christoffel words are the unique balanced words of given length and density, up to rotation \cite{AlloucheShallit2003}. Recall that a binary word is \emph{balanced} if for any two factors or subsequences of the same length, the numbers of ones differ by at most one. They constitute a classical object in combinatorics on words and symbolic dynamics; see \cite{AlloucheShallit2003,Lothaire2002,BerstelDeLuca2002} for systematic accounts.
	
	In the context of the Collatz problem, where the dynamics can be
	encoded by parity vectors \cite{Terras1976,Rajab2022}, it is
	natural to expect that such balanced structures play a relevant
	role in the study of functionals such as $C(d)$. Since $C(d)$
	depends on the distribution of the ones in a non-local manner,
	optimal configurations must avoid local concentrations of ones,
	favouring distributions that are as uniform as possible. This
	leads naturally to Christoffel words as the candidates to maximize
	$C_{\min}(d)$ within each class $\mathcal{D}_{N,r}$.
	
	The main result of this paper, proved in Section~7, confirms this
	expectation: Christoffel words are, up to rotation, the unique
	maximizers of $C_{\min}(d)$ on $\mathcal{D}_{N,r}$.
	
	\section{Local transformations}
	
	We now study the behaviour of $C(d)$ under local transformations
	of the binary word. We begin with a preliminary result.
	
	\begin{prop}\label{prp:concat}
		Let $d_1 \in \mathcal{D}_{N_1}$ and $d_2 \in \mathcal{D}_{N_2}$.
		Then
		\begin{equation}\label{eq:concat}
			C([d_1, d_2]) = 3^{r_0(d_2)}\,C(d_1) + 2^{N_1}\,C(d_2).
		\end{equation}
	\end{prop}
	
	\begin{proof}
		Let $d = [d_1, d_2]$ be the concatenation of both words. Separating
		the sum defining $C(d)$ into the contributions of $d_1$ and $d_2$,
		we obtain
		\begin{equation*}
			C(d) = \sum_{i=1}^{N_1} 2^{i-1} 3^{r_i(d)}\,d_i
			+ \sum_{i=N_1+1}^{N_1+N_2} 2^{i-1} 3^{r_i(d)}\,d_i.
		\end{equation*}
		For $i \le N_1$ we have
		\begin{equation*}
			r_i(d) = r_i(d_1) + r_0(d_2),
		\end{equation*}
		while for $i > N_1$, writing $j = i - N_1$, we have
		\begin{equation*}
			r_i(d) = r_j(d_2).
		\end{equation*}
		Substituting into the expression above gives
		\begin{equation*}
			\sum_{i=1}^{N_1} 2^{i-1} 3^{r_i(d)}\,d_i
			= 3^{r_0(d_2)} \sum_{i=1}^{N_1} 2^{i-1} 3^{r_i(d_1)}\,d_i
			= 3^{r_0(d_2)}\,C(d_1),
		\end{equation*}
		and
		\begin{equation*}
			\sum_{i=N_1+1}^{N_1+N_2} 2^{i-1} 3^{r_i(d)}\,d_i
			= 2^{N_1} \sum_{j=1}^{N_2} 2^{j-1} 3^{r_j(d_2)}\,d_j
			= 2^{N_1}\,C(d_2).
		\end{equation*}
		Therefore $C([d_1,d_2]) = 3^{r_0(d_2)}\,C(d_1) + 2^{N_1}\,C(d_2)$.
	\end{proof}
	
	The concatenation formula \eqref{eq:concat} allows us to compare
	words that differ only in the order of appearance of certain
	symbols, and will be used to establish an order under local
	transpositions.
	
	\begin{prop}\label{prp:swap}
		Let $d_1 \in \mathcal{D}_{N_1}$, $d_2 \in \mathcal{D}_{N_2}$, and
		let
		\begin{equation*}
			d = [d_1, 1, 0, d_2], \qquad d' = [d_1, 0, 1, d_2].
		\end{equation*}
		Then $C(d) < C(d')$.
	\end{prop}
	
	\begin{proof}
		Applying the concatenation formula \eqref{eq:concat}, we have
		\begin{equation*}
			C([d_1, 1, 0, d_2]) = 3^{r_0(d_2)}\,C([d_1,1,0]) + 2^{N_1+2}\,C(d_2),
		\end{equation*}
		and analogously
		\begin{equation*}
			C([d_1, 0, 1, d_2]) = 3^{r_0(d_2)}\,C([d_1,0,1]) + 2^{N_1+2}\,C(d_2).
		\end{equation*}
		It therefore suffices to compare $C([d_1,1,0])$ and $C([d_1,0,1])$.
		Applying \eqref{eq:concat} again,
		\begin{equation*}
			C([d_1,1,0]) = 3\,C(d_1) + 2^{N_1}\,C([1,0]),
			\qquad
			C([d_1,0,1]) = 3\,C(d_1) + 2^{N_1}\,C([0,1]).
		\end{equation*}
		Since $C([1,0]) = 1$ and $C([0,1]) = 2$, we obtain
		\begin{equation*}
			C([d_1,0,1]) - C([d_1,1,0]) = 2^{N_1} > 0.
		\end{equation*}
		Multiplying by $3^{r_0(d_2)} > 0$ gives
		\begin{equation*}
			C(d') - C(d) = 2^{N_1}\,3^{r_0(d_2)} > 0,
		\end{equation*}
		which completes the proof.
	\end{proof}
	
	Proposition~\ref{prp:swap} shows that moving a one to the right
	increases the value of $C$. Consequently, configurations that
	minimize the functional tend to concentrate the ones in earlier
	positions. Moreover, it induces a partial order on the set of
	binary words compatible with the distribution of the ones, in the
	spirit of recent work on parity vectors \cite{Rozier2025}.
	
	This local behaviour is consistent with the global structure of
	Christoffel words, which distribute the ones in a balanced manner
	and therefore appear as natural candidates to maximize $C_{\min}(d)$
	over $\mathcal{D}_{N,r}$.
	
	This partial order will be the key tool in Section~6, where we
	show that any word in $\mathcal{D}_{N,r}$ can be connected to
	the Christoffel word by a finite sequence of $10\to 01$
	transpositions.
	
	\section{Combinatorial structure}
	
	It is natural to study the average position of the ones in a binary
	word. For $d \in \mathcal{D}_{N,r}$ we define
	\begin{equation}\label{eq:mu}
		\mu(d) = \frac{1}{N}\sum_{j=1}^{N} j\, d_j.
	\end{equation}
	The quantity $\mu(d)$ measures the mean position of the ones in $d$.
	Minimizing $\mu$ is therefore equivalent to concentrating the ones
	as far to the left as possible. The key observation connecting $\mu$
	to the functional $C$ is that the rotation minimizing $\mu$ within
	a rotation class is precisely the one that places the ones furthest
	to the left, and this will turn out to be the rotation for which
	$C$ is smallest; working with $\mu$ is more tractable because it
	depends linearly on the positions of the ones, whereas $C$ depends
	on them exponentially.
	
	Rather than restricting to a single rotation class, we consider the
	minimization of $\mu$ over the full set $\mathcal{D}_{N,r}$. This
	relaxation allows us to work in a simpler framework, decoupling the
	analysis from the rotation action and revealing the underlying
	combinatorial structure.
	
	Given $d \in \mathcal{D}_{N,r}$, there exists a rotation
	$d^c = \tau^j(d)$ such that
	\begin{equation}\label{eq:mumin}
		\mu(d^c) = \min_{0 \le j < N} \mu(\tau^j(d)).
	\end{equation}
	The idea is to show that, starting from $d^c$ and applying
	$10 \to 01$ transpositions, we can reach the Christoffel word
	$d^{\mathrm{chr}}_{N,r}$. We begin with the following lemma.
	
	\begin{lemma}\label{lem:balance}
		Let $d \in \mathcal{D}_{N,r}$ and let $d^c = \tau^j(d)$ be a
		rotation of $d$ such that
		\begin{equation*}
			\mu(d^c) = \min_{0 \le j < N} \mu(\tau^j(d)).
		\end{equation*}
		Then
		\begin{equation}\label{eq:balance}
			\frac{1}{N-k}\sum_{j=k+1}^{N} d_j^c \le \frac{r}{N},
			\qquad k = 0, \dots, N-1.
		\end{equation}
	\end{lemma}
	
	\begin{proof}
		We first observe that the condition \eqref{eq:balance} is
		equivalent, for $k \ge 1$, to
		\begin{equation*}
			\frac{1}{k}\sum_{j=1}^{k} d_j^c \ge \frac{r}{N}.
		\end{equation*}
		Suppose for contradiction that there exists $k_0 \in \{1,\dots,N-1\}$
		such that
		\begin{equation*}
			\frac{1}{N-k_0}\sum_{j=k_0+1}^{N} d_j^c > \frac{r}{N},
		\end{equation*}
		and let
		\begin{equation*}
			\tau^{k_0}(d^c) = \left(d_{k_0+1}^c, \dots, d_N^c,
			d_1^c, \dots, d_{k_0}^c\right).
		\end{equation*}
		Then
		\begin{align*}
			\mu(\tau^{k_0}(d^c))
			&= \frac{1}{N}\sum_{j=k_0+1}^{N}(j-k_0)\,d_j^c
			+ \frac{1}{N}\sum_{j=1}^{k_0}(N-k_0+j)\,d_j^c \\
			&= \mu(d^c)
			- \frac{k_0}{N}\sum_{j=k_0+1}^{N} d_j^c
			+ \frac{N-k_0}{N}\sum_{j=1}^{k_0} d_j^c.
		\end{align*}
		By hypothesis,
		\begin{equation*}
			\sum_{j=k_0+1}^{N} d_j^c > \frac{r}{N}(N-k_0),
		\end{equation*}
		and since $\sum_{j=1}^{N} d_j^c = r$, we deduce that
		\begin{equation*}
			\sum_{j=1}^{k_0} d_j^c < \frac{r}{N}\,k_0.
		\end{equation*}
		Substituting gives
		\begin{equation*}
			\mu(\tau^{k_0}(d^c)) < \mu(d^c),
		\end{equation*}
		contradicting the minimality of $d^c$.
	\end{proof}
	
	The inequality \eqref{eq:balance} can be interpreted as a balance condition: no tail of the word has a density of ones exceeding the global density $r/N$. This is a characteristic property of optimal configurations.

    \begin{lemma}\label{lem:positions}
		Let $i_1^{\mathrm{chr}} < \cdots < i_r^{\mathrm{chr}}$ denote the positions of the ones in $d^{\mathrm{chr}}_{N,r}$. Then
		\begin{equation}\label{eq:pos_formula}
			i_j^{\mathrm{chr}} = \floor{\frac{(j-1)N}{r}} + 1, \qquad j = 1, \dots, r.
		\end{equation}
	\end{lemma}

    \begin{proof}
        By definition,
		\begin{equation*}
			i_k^{\mathrm{chr}} = \min\left\{ i:\ceil{\frac{ir}{N}}=k \right\}.
		\end{equation*}
		
		Since $\ceil{\frac{ir}{N}}=k$ is equivalent to 
		\begin{equation*}
			k-1<\frac{ir}{N}\le k,
		\end{equation*}
		the position $i_k^{\mathrm{chr}}$ is given by the smallest integer satisfying
		\begin{equation*}
			i>\frac{(k-1)N}{r}.
		\end{equation*}
		
		Therefore
		\begin{equation*}
			i_k^{\mathrm{chr}} = 1+\floor{\frac{(k-1)N}{r}}.
		\end{equation*}
    \end{proof}
    
	\begin{thm}[Position comparison]\label{teo:maxim}
		Let $d^c$ be a rotation of $d$ minimizing $\mu$ among all rotations
		of $d$, and let $d^{\mathrm{chr}}_{N,r}$ be the Christoffel word in
		$\mathcal{D}_{N,r}$. Denote by $i_k^c$ and $i_k^{\mathrm{chr}}$,
		for $k = 1,\dots,r$, the positions of the $k$-th one in $d^c$ and
		$d^{\mathrm{chr}}_{N,r}$, respectively. Then
		\begin{equation}\label{eq:positions_ineq}
			i_k^c \le i_k^{\mathrm{chr}}, \qquad k = 1,\dots,r.
		\end{equation}
	\end{thm}
	
	\begin{proof}
		Let $d^c=(d_1^c,\ldots,d_N^c)\in \mathcal{D}_{N,r}$ be a rotation
		minimizing $\mu$. By Lemma~\ref{lem:balance},
		\begin{equation*}
			\frac{1}{N-m}\sum_{j=m+1}^{N} d_j^c \le \frac{r}{N},
			\qquad m=0,\ldots,N-1.
		\end{equation*}
		
		Taking $m=i_k^c-1$, and observing that $(d_{i_k^c}^c,\ldots,d_N^c)$ contains exactly $r-k+1$ ones, we obtain
		\begin{equation*}
			\frac{r-k+1}{N-i_k^c+1}\le \frac{r}{N}.
		\end{equation*}
		
		Hence,
		\begin{equation*}
			N(r-k+1)\le r(N-i_k^c+1),
		\end{equation*}
		which yields
		
		Since $i_k^c$ is an integer,
		\begin{equation*}
			i_k^c \le 1+\floor{\frac{(k-1)N}{r}}.
		\end{equation*}
		
	Combining this with Lemma \ref{lem:positions}, we conclude \eqref{eq:positions_ineq}.
	\end{proof}
	
	Inequality \eqref{eq:positions_ineq} expresses that the ones
	in $d^c$ are, in a precise sense, shifted to the left relative to
	those of the Christoffel word. This comparison is consistent with
	the optimal balance of Christoffel words and will allow us to
	connect $d^c$ to $d^{\mathrm{chr}}_{N,r}$ by a chain of
	Proposition~\ref{prp:swap} transpositions.
	
	\begin{thm}[Connection by transpositions]\label{thm:dc_to_chr}
		The word $d^c$ can be transformed into $d^{\mathrm{chr}}_{N,r}$ by
		a finite sequence of $10 \to 01$ transpositions.
	\end{thm}
	
	\begin{proof}
		We proceed by induction, constructing a sequence of words  $d^0, d^1, \dots, d^{r-1}$ with $d^0=d^c$.
        
        The last one in $d^0$ occupies the position $i_{r}^c$. By Theorem~\ref{teo:maxim} $i_{r}^c \leq i_{r}^{\mathrm{chr}}$. Therefore, we can move that last one at position $i_{r}^{\mathrm{chr}}$ to get $d^1$ by a sequence of $i_{r}^{\mathrm{chr}}-i_{r}^c$  transpositions.
        
        Assume that in the word $d^{k}$ the last $k$ ones already occupy their correct positions, that is, they appear at positions $i_{r-k+1}^{\mathrm{chr}}, \dots, i_r^{\mathrm{chr}}$. The remaining $r-k$ ones to the left occupy their original positions, those of the first $r-k$ ones in $d^c$. So, there is a one in $d^{k}$ at position $i_{r-k}^c$.
		
		By Theorem~\ref{teo:maxim}, $i_{r-k}^c \le i_{r-k}^{\mathrm{chr}}$. Moreover, between positions $i_{r-k}^c$ and $i_{r-k}^{\mathrm{chr}}$ there are no ones already fixed, since those occupy positions $i_{r-k+1}^{\mathrm{chr}}, \dots, i_r^{\mathrm{chr}}$, all strictly to the right of $i_{r-k}^{\mathrm{chr}}$. Therefore it is possible to move the one at position $i_{r-k}^c$ to position $i_{r-k}^{\mathrm{chr}}$ by a finite sequence of $10 \to 01$ transpositions, each of which moves it one step to the right without altering any other one. Define $d^{k+1}$ as the word obtained after these transpositions.
		
		Iterating for $k = 1,\dots,r-1$, we obtain $d^{r-1} = d^{\mathrm{chr}}_{N,r}$, which completes the induction.
    \end{proof}
	
	Theorem~\ref{thm:dc_to_chr} shows that $d^{\mathrm{chr}}_{N,r}$
	can be reached from $d^c$ by local transformations that
	progressively redistribute the ones in a more balanced way. This
	reveals the compatibility between the partial order induced by
	$10 \to 01$ transpositions and the structure of Christoffel words.
	
	\section{Main result}
	
	Before establishing the main result, we analyze the behaviour of
	the functional $C$ on Christoffel words as the length increases
	with the number of ones fixed. This controls the growth of $C$
	and will be essential in the final comparison.
	
	\begin{prop}\label{prop:monotone}
		For the Christoffel word $d_{N,r}^{\mathrm{chr}}$ of length $N$
		with $r$ ones, defined by \eqref{eq:Christoffel}, we have
		\begin{equation}\label{eq:growth}
			C(d_{N,r}^{\mathrm{chr}}) < C(d_{N+1,r}^{\mathrm{chr}})
			< 2\,C(d_{N,r}^{\mathrm{chr}}).
		\end{equation}
	\end{prop}
	
	\begin{proof}
		Let $c_k(N) = \ceil{kr/N}$. Then, the characters of the word $d_{N,r}^{\mathrm{chr}}$ can be written as $d_k = c_k(N) - c_{k-1}(N)$,
		and the ones appear exactly at the indices where $c_k(N)$
		increments. Denoting these indices by
		\begin{equation*}
			k_1(N) < \cdots < k_r(N),
		\end{equation*}
		we have
		\begin{equation*}
			C(d_{N,r}^{\mathrm{chr}}) = \sum_{j=1}^{r} 2^{k_j(N)-1}\,3^{r-j}.
		\end{equation*}
		We define the indices $k_j(N+1)$ analogously. Since
		$kr/(N+1) < kr/N$, we have $c_k(N+1) \le c_k(N)$. Both sequences
		are non-decreasing, take integer values in $\{0,\dots,r\}$, and
		can increase by at most one; therefore the positions of the
		increments shift to the right as $N$ increases to $N+1$
		\cite{Lothaire2002,AlloucheShallit2003}. Consequently,
		\begin{equation*}
			k_j(N) \le k_j(N+1),
		\end{equation*}
		with strict inequality for at least one $j$, which gives
		\begin{equation*}
			C(d_{N+1,r}^{\mathrm{chr}}) > C(d_{N,r}^{\mathrm{chr}}).
		\end{equation*}
		On the other hand, we claim that $c_k(N+1) \ge c_{k-1}(N)$, which
		implies $k_j(N+1) \le k_j(N)+1$. Indeed,
		$\frac{kr}{N+1} \ge \frac{(k-1)r}{N}$ is equivalent to
		$kN \ge (k-1)(N+1)$, that is, to $N+1 \ge k$, which holds for
		all $k \le N+1$. Therefore each increment position shifts by at
		most one, so
		\begin{equation*}
			C(d_{N+1,r}^{\mathrm{chr}})
			\le \sum_{j=1}^{r} 2^{k_j(N)}\,3^{r-j}
			= 2\,C(d_{N,r}^{\mathrm{chr}}). \qedhere
		\end{equation*}
	\end{proof}
	
	Proposition~\ref{prop:monotone} shows that $C$ grows in a
	controlled manner as the length increases; the following corollary
	translates this into a monotonicity property of the quantity
	characterizing cycles.
	
	\begin{cor}\label{cor:monotone}
		Assume $2^N > 3^r$, which holds whenever $N/r > \log_2 3$. Then
		\begin{equation}\label{eq:ratio_monotone}
			\frac{C(d_{N+1,r}^{\mathrm{chr}})}{2^{N+1} - 3^{r}}
			< \frac{C(d_{N,r}^{\mathrm{chr}})}{2^{N} - 3^{r}}.
		\end{equation}
	\end{cor}
	
	\begin{proof}
		From Proposition~\ref{prop:monotone},
		$C(d_{N+1,r}^{\mathrm{chr}}) < 2\,C(d_{N,r}^{\mathrm{chr}})$,
		so
		\begin{equation*}
			\frac{C(d_{N+1,r}^{\mathrm{chr}})}{2^{N+1}-3^r}
			< \frac{2\,C(d_{N,r}^{\mathrm{chr}})}{2^{N+1}-3^r}.
		\end{equation*}
		Since $2(2^N - 3^r) = 2^{N+1} - 2 \cdot 3^r < 2^{N+1} - 3^r$ and $2^N > 3^r$ by hypothesis, it follows that
		\begin{equation*}
			\frac{2}{2^{N+1}-3^r} < \frac{1}{2^N - 3^r}.
		\end{equation*}
		Multiplying by $C(d_{N,r}^{\mathrm{chr}}) > 0$ and combining
		with the previous inequality yields \eqref{eq:ratio_monotone}.
	\end{proof}
	
	Corollary~\ref{cor:monotone} shows that as the length increases
	with the number of odd iterates fixed, the quantity associated
	with the possible existence of cycles decreases strictly.
	In particular, longer configurations are progressively less
	favourable for the existence of cycles.

   We are now in a position to establish the main result of the paper.
	
	\begin{thm}[Main theorem]\label{thm:main}
		For every $N \ge 1$ and $0 \le r \le N$,
		\begin{equation}\label{eq:main}
			\max_{d \in \mathcal{D}_{N,r}} C_{\min}(d) = C(d^{\mathrm{chr}}_{N,r}),
		\end{equation}
		where $d^{\mathrm{chr}}_{N,r}$ is the Christoffel word of length $N$ with $r$ ones. Moreover, the maximum is attained uniquely, up to rotation, at $d = d^{\mathrm{chr}}_{N,r}$.
	\end{thm}
	
	\begin{proof}
		Let $d \in \mathcal{D}_{N,r}$ and let $d^c$ be the rotation of $d$ minimizing $\mu$ among all rotations, as in \eqref{eq:mumin}. By definition of $C_{\min}$ and the fact that $d^c$ is one of the rotations over which the minimum is taken,
		\begin{equation*}
			C_{\min}(d) \le C(d^c).
		\end{equation*}
		By Theorem~\ref{thm:dc_to_chr}, $d^c$ can be transformed into $d^{\mathrm{chr}}_{N,r}$ by a finite sequence of $10 \to 01$ transpositions. By Proposition~\ref{prp:swap}, each such transposition strictly increases the value of $C$. Therefore
		\begin{equation*}
			C(d^c) \le C(d^{\mathrm{chr}}_{N,r}),
		\end{equation*}
		with equality if and only if the sequence of transpositions is empty, that is, if and only if $d^c = d^{\mathrm{chr}}_{N,r}$.
        Combining both inequalities,
		\begin{equation*}
			C_{\min}(d) \le C(d^{\mathrm{chr}}_{N,r}),\quad \forall d\in\mathcal{D}_{N,r} \Longrightarrow \max_{d\in\mathcal{D}_{N,r}} C_{\min}(d) \le C(d^{\mathrm{chr}})
		\end{equation*}


        It is now easy to see that the Christoffel word $d^{\mathrm{chr}}_{N,r}$ is the rotation that minimizes $\mu$ within its rotation class, that is
        \begin{equation*}
            (d^{\mathrm{chr}}_{N,r})^{c} = d^{\mathrm{chr}}_{N,r},
        \end{equation*}
        and this follows directly from the previous lemmas.
        
        Moreover, by Theorem~6.3 and Proposition~5.2 we know that there exists, for any $\mu$-minimizing rotation $d^{c}$, a finite sequence of $10 \to 01$ transpositions transforming $d^{c}$ into $d^{\mathrm{chr}}_{N,r}$, each of which strictly increases $C$. In our particular case, starting from $(d^{\mathrm{chr}}_{N,r})^{c} = d^{\mathrm{chr}}_{N,r}$, this sequence of transpositions is empty, so there is no way to obtain a rotation of $d^{\mathrm{chr}}_{N,r}$ with a strictly smaller value of $C$. Consequently,
        \begin{equation*}
            C_{\min}\bigl(d^{\mathrm{chr}}_{N,r}\bigr) = \min_{0 \le j < N} C\bigl(\tau^{j}(d^{\mathrm{chr}}_{N,r})\bigr) = C\bigl(d^{\mathrm{chr}}_{N,r}\bigr),
        \end{equation*}
        which shows that the upper bound $C(d^{\mathrm{chr}}_{N,r})$ is actually attained.
	\end{proof}
	
	Theorem~\ref{thm:main} shows that Christoffel words are not only
	combinatorially balanced, but also constitute the unique extremal
	configurations for the functional $C_{\min}$. This establishes a
	direct link between the dynamics of the Collatz problem and
	classical structures in combinatorics on words, showing that
	optimal configurations are governed by balanced distributions of
	the ones. To the best of our knowledge, this is the first result
	identifying Christoffel words as extremal configurations for a
	functional arising directly from Collatz dynamics.
	
	\section{Bounds}
	
	In this section we use the optimization results obtained previously
	to derive arithmetic restrictions on the possible existence of
	periodic cycles. In particular, we show that the proportion between
	even and odd iterates imposes very rigid constraints on the dynamics.
	
	\begin{thm}\label{thm:NleII2r}
		Let $d \in \mathcal{D}_{N,r}$ be the parity word associated with a
		periodic orbit of the Collatz map. Then:
		\begin{enumerate}
			\item If $N = 2r$, the only periodic orbit with this proportion
			of even and odd iterates is the trivial one.
			\item There is no periodic orbit with $N > 2r$.
		\end{enumerate}
	\end{thm}
	
	\begin{proof}
		\begin{enumerate}
			\item \textbf{Case $N = 2r$.}
			The Christoffel word in this case is
			$d_{2r,r}^{\mathrm{chr}} = [10]^r$, and a direct computation
			gives
			\begin{equation*}
				C(d_{2r,r}^{\mathrm{chr}})
				= \sum_{i=1}^{r} 2^{2i-2}\,3^{r-i}
				= \frac{3^r}{4}\sum_{i=1}^{r}\left(\frac{4}{3}\right)^i
				= 4^r - 3^r.
			\end{equation*}
			By Theorem~\ref{thm:main},
			$C_{\min}(d) \le C(d_{2r,r}^{\mathrm{chr}}) = 4^r - 3^r$
			for every $d \in \mathcal{D}_{2r,r}$. Using \eqref{eq:cycle}
			with $2^N - 3^r = 4^r - 3^r$,
			\begin{equation*}
				x = \frac{C_{\min}(d)}{4^r - 3^r}
				\le \frac{4^r - 3^r}{4^r - 3^r} = 1.
			\end{equation*}
			Since $x \in \mathbb{N}$ and $x \ge 1$, we conclude $x = 1$,
			which corresponds to the trivial cycle. By
			Theorem~\ref{thm:main}, equality holds if and only if $d$ is
			a rotation of $[10]^r$, confirming that the trivial cycle is
			the only one.
			
			\item \textbf{Case $N > 2r$.}
			Write $N = 2r + n_0$ with $n_0 > 0$. Applying
			Corollary~\ref{cor:monotone} repeatedly and noting that
			$2^N > 3^r$ for all steps since $N/r > 2 > \log_2 3$,
			\begin{equation*}
				\frac{C(d_{N,r}^{\mathrm{chr}})}{2^N - 3^r}
				< \cdots
				< \frac{C(d_{2r,r}^{\mathrm{chr}})}{4^r - 3^r} = 1.
			\end{equation*}
			By Theorem~\ref{thm:main}, $C_{\min}(d) \le
			C(d_{N,r}^{\mathrm{chr}})$, so
			\begin{equation*}
				x = \frac{C_{\min}(d)}{2^N - 3^r}
				\le \frac{C(d_{N,r}^{\mathrm{chr}})}{2^N - 3^r} < 1,
			\end{equation*}
			contradicting $x \in \mathbb{N}$ with $x \ge 1$.
		\end{enumerate}
	\end{proof}
	
	Theorem~\ref{thm:NleII2r} shows that the density of odd iterates
	in a periodic orbit is strongly constrained. The bound $N \le 2r$
	is known in the literature \cite{Lagarias1985}; what is new here
	is that it follows directly from the extremality of Christoffel
	words established in Theorem~\ref{thm:main}. Together with the
	necessary condition
	\begin{equation*}
		\frac{N}{r} > \frac{\log 3}{\log 2} \approx 1.585,
	\end{equation*}
	this gives the sharp bound on the slope:
	\begin{equation}\label{eq:slope_bound}
		\frac{\log 3}{\log 2} < \frac{N}{r} \le 2.
	\end{equation}
	
	To obtain an explicit bound on $x$ in terms of $N/r$, we need
	to explicitly estimate $C(d_{N,r}^{\mathrm{chr}})$. This requires
	knowing the positions of the ones, which is given by Lemma~\ref{lem:positions}. The estimate in Theorem~\ref{thm:bound} will be
	sharp when the floor function in \eqref{eq:pos_formula} is close to the
	continuous value $(j-1)N/r$, which occurs precisely for
	Christoffel words.
	
	\begin{thm}\label{thm:bound}
		Let $d \in \mathcal{D}_{N,r}$ be the parity word associated with
		a periodic orbit of the Collatz map, with $N/r > \log_2 3$. Then
		there exists an element $x$ of the orbit satisfying
		\begin{equation}\label{eq:bound}
			x \le \frac{1}{2^{N/r} - 3}.
		\end{equation}
	\end{thm}
	
	\begin{proof}
		By Theorem~\ref{thm:main} and \eqref{eq:cycle},
		\begin{equation}
        \label{eq:bound_aux_1}
			x = \frac{C_{\min}(d)}{2^N - 3^r}
			\le \frac{C(d_{N,r}^{\mathrm{chr}})}{2^N - 3^r}.
		\end{equation}
		Using Lemma~\ref{lem:positions} and \eqref{eq:C},
        \begin{equation*}
			C(d_{N,r}^{\mathrm{chr}})
			= \sum_{j=1}^{r} 2^{i_j^{\mathrm{chr}}-1}\,3^{r-j}
			= \sum_{j=1}^{r} 2^{\floor{(j-1)N/r}}\,3^{r-j}.
		\end{equation*}
		Since $\floor{(j-1)N/r} \le (j-1)N/r$, we have
		$2^{\floor{(j-1)N/r}} \le 2^{(j-1)N/r}$, and therefore
		\begin{equation*}
			C(d_{N,r}^{\mathrm{chr}})
			\le \sum_{j=1}^{r} 2^{(j-1)N/r}\,3^{r-j}
			= 3^{r-1} \sum_{j=1}^{r} \left(\frac{2^{N/r}}{3}\right)^{j-1}.
		\end{equation*}
		This is a geometric series with ratio $q = 2^{N/r}/3 > 1$
		(since $N/r > \log_2 3$), giving
		\begin{equation*}
			C(d_{N,r}^{\mathrm{chr}})
			\le
            3^{r-1} \sum_{j=1}^{r} q^{j-1}
			= 3^{r-1} \frac{q^r - 1}{q - 1}
			= 3^{r-1} \frac{2^N/3^r - 1}{2^{N/r}/3 - 1}
			= \frac{2^N - 3^r}{2^{N/r} - 3}.
		\end{equation*}
		Using this bound in \eqref{eq:bound_aux_1} we get \eqref{eq:bound}.
	\end{proof}
	
	The bound \eqref{eq:bound} translates the combinatorial structure
	of parity sequences into an explicit arithmetic constraint on the
	minimum element of a periodic orbit. Since $r \mapsto 2^{N/r}$ is
	decreasing in $r$, the bound \eqref{eq:bound} is largest when $r$
	is largest, that is, when $r$ takes its maximum admissible value
	\begin{equation*}
		r_0 = \floor{N\,\frac{\log 2}{\log 3}}.
	\end{equation*}

	\begin{cor}\label{cor:universal}
		For every $N \ge 1$, the minimum element of any periodic orbit of
		length $N$ satisfies
		\begin{equation}\label{eq:universal_bound}
			x \le \frac{1}{2^{N/r_0} - 3},
			\qquad r_0 = \floor{N\,\frac{\log 2}{\log 3}},
		\end{equation}
		a bound depending only on $N$.
	\end{cor}
	
	\begin{proof}
		Apply Theorem~\ref{thm:bound} with $r = r_0$, noting that
		$r \mapsto (2^{N/r} - 3)^{-1}$ is decreasing since
		$r \mapsto 2^{N/r}$ is decreasing, so the bound is maximized
		at the largest admissible $r$, which is $r_0$.
	\end{proof}

    Bound \eqref{eq:universal_bound} is particularly useful for computational searches, as it provides an a priori upper bound on the minimum element of any cycle of length $N$, independently of the number of odd iterates. Table \ref{tab:universal} lists the values of this universal bound for a selection of periods $N \le 485$.
	
	\begin{table}[ht]
		\centering
		\caption{Universal bound \eqref{eq:universal_bound} on the
			minimum element of any periodic orbit of length $N$, for
			selected values of $N$. Here $r_0 = \lfloor N\log 2/\log 3
			\rfloor$ is the worst-case number of odd iterates. The
			non-monotone behaviour of the bound is a consequence of the
			irregular approximation of $\log_2 3$ by rationals. The last two rows correspond to values of $N$ for which
			$r_0/N$ is an exceptionally good rational approximation
			of $\log 2/\log 3$, producing very large bounds and
			illustrating the divergence of the universal bound as
			$N/r_0 \to \log_2 3^+$.}
		\label{tab:universal}
		\smallskip
		\begin{tabular}{rrrr}
			\hline\noalign{\smallskip}
			$N$ & $r_0$ & $N/r_0$ & $\dfrac{1}{2^{N/r_0}-3}$\\[6pt]
			\hline\noalign{\smallskip}
			10 &   6 & $1.6667$ & $      5.721$ \\
			20 &  12 & $1.6667$ & $      5.721$ \\
			30 &  18 & $1.6667$ & $      5.721$ \\
			40 &  25 & $1.6000$ & $     31.814$ \\
			50 &  31 & $1.6129$ & $     17.045$ \\
			60 &  37 & $1.6216$ & $     12.952$ \\
			70 &  44 & $1.5909$ & $     80.703$ \\
			80 &  50 & $1.6000$ & $     31.814$ \\
			90 &  56 & $1.6071$ & $     21.515$ \\
			100 &  63 & $1.5873$ & $   205.426$ \\
			150 &  94 & $1.5957$ & $    44.435$ \\
			200 & 126 & $1.5873$ & $   205.426$ \\
			306 & 193 & $1.5855$ & $   907.656$ \\
			485 & 306 & $1.5850$ & $ 99780.791$ \\
			\noalign{\smallskip}\hline
		\end{tabular}
	\end{table}
	
	\section{Discussion}
	
	The results obtained allow us to reinterpret the problem of the
	existence of nontrivial cycles of the Collatz map in purely
	combinatorial terms, reducing it to the study of binary words of
	fixed length and prescribed density.
	
	The main theorem shows that, within each class $\mathcal{D}_{N,r}$,
	the Christoffel word $d^{\mathrm{chr}}_{N,r}$ is the unique
	maximizer, up to rotation, of $C_{\min}(d)$. This establishes a
	direct connection between Collatz dynamics and the classical theory
	of balanced words, where Christoffel words appear as extremal
	configurations characterized by an optimal distribution of symbols.
	From this perspective, the existence of periodic cycles is
	constrained by
	\begin{equation*}
		x = \frac{C_{\min}(d)}{2^N - 3^r}
		\le \frac{C(d^{\mathrm{chr}}_{N,r})}{2^N - 3^r},
	\end{equation*}
	which reduces the analysis essentially to understanding the
	behaviour of Christoffel words.
	
	As a first consequence, no cycle can satisfy $N > 2r$, and the
	critical case $N = 2r$ corresponds exclusively to the trivial
	cycle. As noted in Section~8, the bound $N \le 2r$ is known in
	the literature \cite{Lagarias1985}; the contribution of the
	present work is to derive it as a direct consequence of the
	extremality of Christoffel words, placing it within a broader
	combinatorial framework. Combined with the necessary condition
	$N/r > \log_2 3$, this gives the sharp constraint
	\eqref{eq:slope_bound} on the admissible slopes.
	
	The explicit bound
	\begin{equation*}
		x \le \frac{1}{2^{N/r} - 3}
	\end{equation*}
	provides a direct relation between the length of the orbit and
	the size of its minimum element, giving effective control over
	the possible candidates for cycles in terms of the parameters
	$(N, r)$. The worst case occurs when $r \approx N \log 2 / \log 3$,
	which coincides with the critical slope arising naturally in
	heuristic models of the Collatz problem
	\cite{Lagarias1985,KontorovichLagarias2010}, and leads to the
	universal bound of Corollary~\ref{cor:universal}.
	
	From a conceptual point of view, these results show that the
	combinatorial structure of parity sequences imposes strong
	restrictions on the possible existence of cycles. In particular,
	the compatibility between the partial order induced by
	$10 \to 01$ transpositions and the structure of Christoffel words
	suggests that extremal configurations exhibit a pronounced
	structural rigidity: any word that is not a rotation of a
	Christoffel word is strictly suboptimal for $C_{\min}$, and
	therefore cannot correspond to a cycle with parameters close to
	the critical values.
	
	Several directions remain open. First, the bound \eqref{eq:bound}
	could potentially be sharpened by a more precise analysis of the
	gap between $\floor{(j-1)N/r}$ and $(j-1)N/r$, which depends on
	the continued fraction expansion of $r/N$. Second, the role of
	Sturmian words --- the aperiodic analogues of Christoffel words,
	corresponding to irrational slopes --- in the study of unbounded
	orbits deserves further investigation: if the density of odd
	iterates along an orbit converges to an irrational value, the
	associated parity sequence approaches a Sturmian word, and the
	extremal properties of such sequences may impose constraints on
	the growth of the orbit. Third, the framework developed here
	could be extended to variants of the Collatz map, such as the
	$ax + b$ family, or to other discrete dynamical systems with
	symbolic structure admitting a similar functional $C(d)$.

After this work was completed, we became aware of the recent article by Kevin Knight \cite{Knight2026}, which studies rational Collatz cycles and identifies upper Christoffel words in that setting. Although both works highlight the role of Christoffel words in Collatz dynamics, they address different but complementary questions.

Knight considers the case of rational Collatz cycles with odd denominators, showing that upper Christoffel words parametrize the high cycles of prescribed length and odd density, and proving that none of these high cycles can consist entirely of integers. By contrast, we study the classical accelerated Collatz map on $\mathbf{N}$ and formulate a discrete optimization problem for the rotation-invariant functional $C_{\min}$ on the space $D_{N,r}$ of parity words.

Our Theorem~\ref{thm:main} proves that, for every admissible pair $(N,r)$, the corresponding Christoffel word is the unique maximizer of $C_{\min}$ up to rotation. Theorem~\ref{thm:NleII2r} then translates this extremal property into explicit bounds on the existence and size of integer cycles. Together with Knight's non-integrality result for rational high cycles, our results show that, whenever $\frac{N}{r}\in\left(\log_2 3,2\right]$, the unique extremal parity pattern cannot correspond to an entirely integer cycle. Consequently, any hypothetical integer cycle must arise from a strictly suboptimal parity pattern, to which our upper bounds impose additional constraints.
	
\end{document}